\documentclass[a4paper,12pt]{amsart}
\usepackage{amssymb,amscd,amsmath,amsthm}
\usepackage{indentfirst,latexsym,bm}
\usepackage{multirow}
\usepackage{enumerate}
\usepackage{setspace}
\usepackage{color}

\theoremstyle{plain}
\newtheorem{theorem}{Theorem}[section]
\newtheorem{corollary}[theorem]{Corollary}
\newtheorem{lemma}[theorem]{Lemma}
\newtheorem{proposition}[theorem]{Proposition}

\theoremstyle{definition}
\newtheorem{definition}[theorem]{Definition}

\newtheorem{remark}[theorem]{Remark}

\numberwithin{equation}{section}

\newcommand{\ds}{\displaystyle}

\newcommand{\Real}{\mathbb{R}}

\newcommand{\Rational}{\mathbb{Q}}

\newcommand{\Integer}{\mathbb{Z}}
\newcommand{\Sphere}{S}
\let\Contentsline\contentsline
\renewcommand

\begin{document}

\title{More on spherical designs of harmonic index $t$}

\author{Yan Zhu}
\address{Department of Mathematics, Shanghai Jiao Tong University, 800 Dongchuan Road, Minhang District, Shanghai, China}
\email{zhuyan@sjtu.edu.cn}

\author{Eiichi Bannai}
\address{Department of Mathematics, Shanghai Jiao Tong University, 800 Dongchuan Road, Minhang District, Shanghai, China}
\email{bannai@sjtu.edu.cn (and bannai@math.kyushu-u.ac.jp)}

\author{Etsuko Bannai}
\address{Misakigaoka 2-8-21, Itoshima-shi, Fukuoka, 819-1136, Japan}
\email{et-ban@rc4.so-net.ne.jp}

\author{Kyoung-Tark Kim}
\address{Department of Mathematics, Shanghai Jiao Tong University, 800 Dongchuan Road, Minhang District, Shanghai, China}
\email{kyoungtarkkim@gmail.com}

\author{Wei-Hsuan Yu}
\address{ Department of Mathematics, Michigan State University, 619 Red Cedar Road, East Lansing, MI 48824}
\email{u690604@gmail.com}

\maketitle

% \footnotetext[1]{\color{red} The first author and the corresponding author}

% --------------------------------------------------------------------------------
\begin{abstract}
A finite subset $Y$ on the unit sphere $\Sphere^{n-1} \subseteq \mathbb{R}^n$ is called a spherical design of harmonic index $t$, if the following condition is satisfied:
$\sum_{\mathbf{x}\in Y}f(\mathbf{x})=0$ for all real homogeneous harmonic polynomials $f(x_1,\ldots,x_n)$ of degree $t$.
Also, for a subset $T$ of $\mathbb{N} = \{1,2,\cdots \}$, a finite subset $Y\subset \Sphere^{n-1}$ is called a spherical design of harmonic index $T,$ if $\sum_{\mathbf{x}\in Y}f(\mathbf{x})=0$ is satisfied for all real homogeneous harmonic polynomials $f(x_1,\ldots,x_n)$ of degree $k$ with $k\in T$.

In the present paper we first study Fisher type lower bounds for the sizes of spherical designs of harmonic index $t$ (or for harmonic index $T$).
We also study 'tight' spherical designs of harmonic index $t$ or index $T$.
Here 'tight' means that the size of $Y$ attains the lower bound for this Fisher type inequality.
The classification problem of tight spherical designs of harmonic index $t$ was started by Bannai-Okuda-Tagami (2015), and the case $t = 4$ was completed by Okuda-Yu (2015+).
In this paper we show the classification (non-existence) of tight spherical designs of harmonic index 6 and 8, as well as the asymptotic non-existence of tight spherical designs of harmonic index $2e$ for general $e\geq 3$.
We also study the existence problem for tight spherical designs of harmonic index $T$ for some $T$, in particular, including index $T = \{8,4\}$.
We use (i) the linear programming method by Delsarte, (ii) the detailed information on the locations of the zeros as well as the local minimum values of Gegenbauer polynomials, (iii) the generalization by Hiroshi Nozaki of the Larman-Rogers-Seidel theorem on $2$-distance sets to $s$-distance sets, (iv) the theory of elliptic diophantine equations, and (v) the semidefinite programming method of eliminating some $2$-angular line systems for small dimensions.\\
\end{abstract}
% --------------------------------------------------------------------------------

\setcounter{tocdepth}{1}
%\tableofcontents

% \begin{itemize}
% \item [] Table of contents\\
% \item [$\S1$] Introduction and spherical designs of harmonic index $t$ (or $T$)
% \item [$\S2$] Linear programming method for spherical designs of harmonic index $t$ (or $T$)
% \item [$\S3$] Fisher type lower bounds and tight spherical designs of harmonic index $t$ (or $T$)
% \item [$\S4$] The non-existence of tight spherical designs of harmonic index 6 and 8
% \item [$\S5$] The asymptotic non-existence of tight spherical designs of harmonic index $2e$ for general $e$
% \item [$\S6$] Tight spherical designs of harmonic index $\{8,4\}$
% \item [$\S7$] Tight spherical designs of harmonic index $\{8,2\}$, $\{8,6\}$, $\{6,2\}$, $\{6,4\},$ as well as tight spherical designs of harmonic index $\{10,6,2\}$, $\{12,8,4\}$
% \item [$\S8$]  Concluding remarks
% \item [] References
% \end{itemize}

Keywords: spherical design, spherical designs of harmonic index,
 Gegenbauer polynomial, Fisher type lower bound, tight design,
 Larman-Rogers-Seidel's theorem, Delsarte's method, linear
 programming, semidefinite programming, elliptic diophantine equation

\section{Introduction and spherical designs of harmonic index $t$ (or $T$)}
\label{sec:introduction}

\noindent
Throughout this paper $Y$ is assumed to be a finite non-empty set, and we denote the set of positive (resp. non-negative) integers by $\mathbb{N}$ (resp. $\mathbb{N}_0$).

Let ${\Sphere^{n-1} = \{\mathbf{x} = (x_1,\cdots ,x_n)\in \mathbb{R}^n \; \mid\; x_1^2 + \cdots +x_n^2=1\}}$ be the unit sphere in the Euclidean space $\mathbb{R}^n.$
Delsarte--Goethals--Seidel \cite[Definition 5.1]{DGS1977} (1977) gave the following definition of spherical designs.

\begin{definition}[Spherical $t$-designs]\label{dfe:1.1}
Let $ {t\in \mathbb{N}_0}$.
A subset $Y\subset \Sphere^{n-1}$ is called a \textit{spherical $t$-design} on $\Sphere^{n-1}$, if and only if
\begin{equation}\label{spht}
\frac{1}{\left|\Sphere^{n-1}\right|}\int_{\mathbf{x}\in\Sphere^{n-1}}f(\mathbf{x})d\sigma(\mathbf{x})=\frac{1}{\left|Y\right|}\sum_{\mathbf{x}\in Y}f(\mathbf{x})
\end{equation}
for any real polynomial $f(x_1,\ldots ,x_n)$ of degree at most $t$, where $\left|\Sphere^{n-1}\right|$ denotes the volume (namely the surface area) of the sphere $\Sphere^{n-1},$
and the integral is the surface integral on $\Sphere^{n-1}.$
\end{definition}

The condition~\eqref{spht} is known to be equivalent to the condition:
\[ \sum_{\mathbf{x}\in Y}f(\mathbf{x})=0 \; \mbox{ for all }  f(x_1, \ldots, x_n)\in \mathrm{Harm}_k^n,\; 1\leq k\leq t\, , \]
where $\mathrm{Harm}_k^n$ is the space of real homogeneous harmonic polynomials of degree $k$ in $n$ indeterminates.

In connection with the latter equivalent defining condition for spherical $t$-designs, we define a weaker concept which we call designs of harmonic index $t$ as follows.

\begin{definition}[Spherical designs of harmonic index $t$]\label{dfe:1.2}
A subset $Y\subset  \Sphere^{n-1}$ is called a \textit{spherical design of harmonic index $t$} on $\Sphere^{n-1}$, if and only if
\begin{equation}\label{equ:star} \tag{$\ast$}
\sum_{\mathbf{x}\in Y}f(\mathbf{x})=0
\end{equation}
for all real homogeneous harmonic polynomial $f(x_1,\ldots ,x_n)$ of degree exactly $t$.
\end{definition}

More generally we have the following definition.

\begin{definition}[Spherical designs of harmonic index $T$]\label{dfe:1.3}
Let $T$ be a subset of $\mathbb{N}$.
A subset $Y\subset  \Sphere^{n-1}$ is called a \textit{spherical design of harmonic index $T$} on $\Sphere^{n-1}$, if and only if
\[\sum_{\mathbf{x}\in Y}f(\mathbf{x})=0 \; \mbox{ for all } f(x_1,\ldots,x_n) \in \mathrm{Harm}_k^n, \mbox{ with } k\in T.\]
\end{definition}

The case of $T$-design with $T=\{1,2,\ldots ,t\}$ corresponds to a usual spherical $t$-design, and the case $T=\{t\}$ corresponds to a spherical design of harmonic index $t$.

The purpose of this paper is to study spherical designs of harmonic index $t$ as well as harmonic index $T$ for some $T$, and convince the reader that these are interesting mathematical objects.
% The purpose of this paper is to push the study of spherical designs of harmonic index $t,$ started in Bannai--Okuda--Tagami \cite{BOT2015} further more.
Our main concerns are Fisher type lower bounds for spherical designs of harmonic index $t$ and $T$, as well as the classification problems of so called `tight' designs.
Here `tight' means those that satisfy the lower bound in a Fisher type inequality.
In Section~\ref{sec:LP} we present general observation about the linear programming method for spherical designs of harmonic index $T$.
In Section~\ref{sec:Fisherandtight} we formulate Fisher type inequalities and tight spherical designs of harmonic index $t$ and $T$.
In the subsequent sections we will study some specific problems.
In Section~\ref{sec:nonexistence} the complete non-existence of tight spherical designs of harmonic index 6 and 8 is proved.
(Note that case of $t=4$ was already settled by Okuda--Yu \cite{OY2015} in a beautiful way by applying the SDP (semidefinite programming) to the existence problem of equiangular lines.
Also note that our proofs for $t=6$ and $t=8$ are obtained in an elementary level without recourse to such deeper consideration as SDP.)
In Section~\ref{sec:asymptotic} we show the asymptotic non-existence of harmonic index $2e$ case for general $e\geq 3$.
Then we turn our attention to the case of $T=\{t_1, t_2\}$.
The center of our study is for $T=\{8,4\}$ in Section~\ref{sec:84}.
In Section~\ref{sec:several} we study the cases $T= \{8,2\}$, $\{8,6\}$, $\{6,2\}$, $\{6,4\}$, as well as $\{10,6,2\}$, and $\{12,8,4\}$.
We conclude the paper in Section~\ref{sec:remark} by mentioning some concluding remarks.

As we mentioned in Abstract, the technique we used are:
(i) the linear programming method by Delsarte, (ii) the detailed information on the locations of the zeros as well as the local minimum values of Gegenbauer polynomials, (iii) the generalization by Hiroshi Nozaki of the Larman-Rogers-Seidel theorem on $2$-distance sets to $s$-distance sets, (iv) the theory of elliptic diophantine equations, and (v) the semidefinite programming method of eliminating some $2$-angular line systems for small dimensions.

\section{Linear programming method for spherical designs of harmonic index $T$}
\label{sec:LP}

\noindent
In this section we consider the linear programming method for spherical designs of harmonic index $T$.
We also introduce basic terminology and notation which will be used in the subsequent sections.

Let $Q_{n,k}(x)$ be the \textit{Gegenbauer polynomial} of degree $k$ in one indeterminate $x$ as introduced in \cite[Definition 2.1]{DGS1977}.
Recall how $Q_{n,k}(x)$ are normalized \cite[Theorem 2.4, Theorem 3.2]{DGS1977}:
\[ Q_{n,k}(1) = \mathrm{dim}\, \mathrm{Harm}_k^n = \binom{n+k-1}{n-1} - \binom{n+k-3}{n-1} =: h_k^n \, . \]
The Gegenbauer polynomials $Q_{n,k}(x)$ are orthogonal polynomials on the closed interval $[-1,1]$ with the weight function $(1-x^2)^{(n-3)/2}$, i.e.,
\[ \int_{-1}^{1} Q_{n,k}(x) \, Q_{n,\ell}(x) \, (1-x^2)^{\frac{n-3}{2}}\; dx = a_{n,k} \delta_{k,\ell} \, ,\]
where $a_{n,k}$ is some constant depending on $n$ and $k$, and $\delta_{k,\ell}$ is the Kronecker delta.
From this orthogonality it can be well established that to any real polynomial $F(x)$ of degree $r$ we can associate its \textit{Gegenbauer expansion}
\begin{equation}\label{eq:Gexp}
 F(x) = \sum_{k=0}^{r} f_k \, Q_{n,k}(x)\, ,
\end{equation}
where the \textit{Gegenbauer coefficients $f_k$} are as follows:
\[ f_k = \frac{1}{a_{n,k}} \int_{-1}^{1} F(x)\, Q_{n,k}(x)\, (1-x^2)^{\frac{n-3}{2}} \; dx.\]

We denote by $\mathbf{x}\cdot \mathbf{y}$ the standard inner product of $\mathbf{x}$ and $\mathbf{y}$ in $\mathbb{R}^n$.
For a subset $Y \subseteq \mathbb{R}^n$ we set $I(Y) := \{ \mathbf{x}\cdot \mathbf{y} \, : \, \mathbf{x},\mathbf{y} \in Y,\; \mathbf{x}\neq \mathbf{y}\}$.
If $\{ e_{k,1}, \ldots, e_{k,h_k^n} \}$ is an orthonormal basis for $\mathrm{Harm}_k^n$ with respect to the inner product $\langle f,g\rangle = \frac{1}{\left|\Sphere^{n-1}\right|} \int_{\mathbf{x}\in\Sphere^{n-1}} f(\mathbf{x})g(\mathbf{x}) d\sigma(\mathbf{x})$ in $\mathrm{Harm}_k^n$, then the \textit{addition formula} says, for any $\mathbf{x},\mathbf{y} \in \Sphere^{n-1}$,
\[ Q_{n,k}(\mathbf{x}\cdot \mathbf{y}) = \sum_{i=1}^{h_k^n} e_{k,i}(\mathbf{x}) e_{k,i}(\mathbf{y}).\]
From the addition formula we have, for any subset $Y\subseteq \Sphere^{n-1}$,
\[ \begin{array}{rcl}
M_k(Y) & := & {\displaystyle \sum_{\mathbf{x},\mathbf{y} \in Y} Q_{n,k}(\mathbf{x}\cdot \mathbf{y}) = \sum_{\mathbf{x},\mathbf{y} \in Y} \sum_{i=1}^{h_k^n} e_{k,i}(\mathbf{x}) e_{k,i}(\mathbf{y}) } \\
& & \\
 & = & {\displaystyle \sum_{i=1}^{h_k^n} \sum_{\mathbf{x},\mathbf{y} \in Y} e_{k,i}(\mathbf{x}) e_{k,i}(\mathbf{y}) = \sum_{i=1}^{h_k^n} \left( \sum_{\mathbf{x} \in Y} e_{k,i}(\mathbf{x}) \right)^2 }. \\
\end{array} \]
Thus we obtain (see Definition~\ref{dfe:1.3} for (M2)) the following two simple observations:
\begin{enumerate}
\item [(M1)] The quantity $M_k(Y)$ is always non-negative;
\item [(M2)] Moreover, $M_k(Y) = 0$ if and only if $Y \subseteq \Sphere^{n-1}$ is a spherical design of harmonic index $k$.
\end{enumerate}

We introduce an identity (see \eqref{eq:mi} below) which turns out to be the main source of Fisher type inequalities.
(See \cite{DGS1977} for the original discussion about so-called `\textit{linear programming bounds}' for spherical designs.)
Suppose $F(x)$ is a non-constant real polynomial of degree $r$ which is of the form \eqref{eq:Gexp}.
For a subset $Y \subseteq \Sphere^{n-1}$, if we calculate $\sum_{\mathbf{x},\mathbf{y} \in Y} F(\mathbf{x}\cdot \mathbf{y})$ in two different ways, then we find
\begin{equation}\label{eq:mi}
\left| Y \right| F(1) + \sum_{\mathbf{x},\mathbf{y}\in Y, \atop \mathbf{x}\neq \mathbf{y}} F(\mathbf{x}\cdot \mathbf{y}) = \left| Y \right|^2 f_0 + \sum_{k=1}^r f_k \, M_k(Y).
\end{equation}

Now suppose $Y \subseteq \Sphere^{n-1}$ is a spherical design of harmonic index $T$.
Observe that if $Y_1, Y_2 \subseteq \Sphere^{n-1}$ are spherical designs of harmonic index $T$ with $Y_1 \cap Y_2 = \emptyset$, then so is $Y_1 \cup Y_2$.
Thus we are interested in finding a lower bound for $\left| Y \right|$.
If $F(x)$ satisfies
\begin{enumerate}
\item [(LP1)] $F(u) \geq 0$ for each $u \in [-1,1]$;
\item [(LP2)] $f_k \leq 0$ for each $k \in \{1,\ldots, r\} \setminus T$,
\end{enumerate}
then we obtain (recall (M1) and (M2)) that
\begin{equation}\label{eq:lpe}
\left| Y \right| F(1) \leq \mbox{L.H.S. of \eqref{eq:mi}}=\mbox{R.H.S. of \eqref{eq:mi}} \leq \left| Y \right|^2 f_0 \, ,
\end{equation}
where the first and second inequalities are due to (LP1) and (LP2), respectively.
Moreover we automatically have $f_0 >0$, since by (LP1) the integrand of the following integral
\[ f_0 = \frac{1}{a_{n,0}} \int_{-1}^{1} F(x) \, (1-x^2)^{\frac{n-3}{2}}\; dx \]
is non-negative, and $a_{n,0} = \int_{-1}^{1} (1-x^2)^{\frac{n-3}{2}}\; dx$ is also positive.
Finally we conclude from \eqref{eq:lpe}
\begin{equation}\label{eq:mie}
\left| Y \right| \geq \frac{F(1)}{f_0}.
\end{equation}

The inequality \eqref{eq:mie} leads to the following natural question:
What is the optimal choice of $F(x)$?
In other words, what $F(x)$ guarantees that $F(1) / f_0$ becomes largest?
However, this problem is normally not easy to solve.
If the degree $r$ of $F(x)$ is not determined, then our linear programming is in fact infinite, i.e., the variables are $f_0 > 0$, $f_1, f_2, \ldots, u$ ($u$ is the variable appeared in (LP1)).
Even if we consider $F(x)$ of fixed degree, this problem is not easy in general.

Another important question is as follows:
When the equality for $\left| Y \right| \geq F(1)/f_0$ holds?
We know from \eqref{eq:mi} and \eqref{eq:lpe} that the equality holds if and only if $\sum_{\mathbf{x},\mathbf{y}\in Y,\mathbf{x}\neq \mathbf{y}} F(\mathbf{x}\cdot \mathbf{y}) = 0$ and $\sum_{k=1}^r f_k \, M_k(Y) = 0$ if and only if the following two conditions hold:
\begin{enumerate}
\item [(E1)] $F(\alpha) = 0$ for all $\alpha \in I(Y)$;
\item [(E2)] For each $k \in \{1, \ldots, r\} \setminus T$, if $M_k(Y) \neq 0$ then $f_k = 0$.
\end{enumerate}

The next observation is based upon the preceding discussion.
\begin{proposition}\label{ourtestfunction}
Let $Y$ be a spherical design of harmonic index $T$ on $\Sphere^{n-1}$.
Suppose that $F(x)$ satisfies $\mathrm{(LP1)}$ and $\mathrm{(LP2)}$, and has the following form:
\[ F(x) = c + \sum_{k \in T} f_k \, Q_{n,k}(x) , \]
where $c$ is a constant.
Then the equality for \eqref{eq:mie} holds if and only if $F(\alpha) = 0$ for all $\alpha \in I(Y)$, i.e., $I(Y)$ is contained in the set of roots of $F(x)$.
\end{proposition}

\section{Fisher type lower bounds and tight spherical designs of harmonic index $t$ (or $T$)}
\label{sec:Fisherandtight}

\noindent
In the first section we defined the concept of spherical designs of harmonic index $t$ or more generally for $T$.
This notion was already essentially defined in the literature, as ``a spherical design which admits indices $T$''.
(See Delsarte-Seidel \cite{DS1989}, say.)
On the other hand, the terminology of spherical design of index $t$ is already defined as ``a spherical design for which the equality \eqref{equ:star} holds for any \emph{homogeneous harmonic polynomials of degree t}'', say in \cite{DS1989}, \cite{LV1993}, etc.
In order to avoid the confusion with these terminologies, we use the term `spherical designs of harmonic index $t$ (or $T$)'.
It seems that no systematic study of spherical designs of harmonic index $t$ has been made, before Bannai--Okuda--Tagami \cite{BOT2015}.
In \cite[Theorem 1.2]{BOT2015}, a Fisher type lower bound was obtained for spherical designs of harmonic index $t$.

What would be the natural Fisher type lower bound for spherical designs of harmonic index $T$?
We propose the following approach.
Suppose $T=\{t_1 (=2e),t_2, \ldots ,t_{\ell} \}$ with $t_1 > \cdots >t_{\ell}$.
Consider the linear combination
\[ L(x) = Q_{n,2e}(x) + f_{t_2}Q_{n,t_2}(x) + \cdots + f_{t_{\ell}}Q_{n,t_{\ell}}(x), \]
and let $L(x)$ take the minimum value $-c_{n,T}$ only at $\ell$ non-negative points.
It is not clear when such coefficients $f_{t_2}, \ldots, f_{t_{\ell}}$ exist, but we are interested in the case where they exist.
If the coefficients $f_{t_2}, \ldots, f_{t_{\ell}} $ are uniquely determined, then we take (with such coefficients)
\[  F(x) = c_{n,T} + L(x) \quad \mbox{ where } c_{n,T} := - \min L(x).\]
Since this $F(x)$ clearly satisfies the conditions (LP1) and (LP2) in Section~\ref{sec:LP}, we obtain
\[ \left| Y \right| \geq \frac{F(1)}{c_{n,T}} =: b_{n, T}.\]
In this paper we call $Y\subset \Sphere^{n-1}$ a tight spherical design of harmonic index $T$, if the equality is attained in the above Fisher type inequality.
We emphasize that our definition of tight designs of harmonic index $T$ is a conventional definition, but we believe this definition is still meaningful.

\section{The non-existence of tight spherical designs of harmonic index 6 and 8}
\label{sec:nonexistence}

\noindent
In this section we will prove the non-existence of tight spherical designs of harmonic index $t=6$ and $t=8$.

Bannai--Okuda--Tagami \cite{BOT2015} (2015) gave Fisher type lower bounds for spherical designs of harmonic index $t$.
\begin{theorem}[{\cite[Theorem 1.2]{BOT2015}}]\label{theo:2.1}
Let $Y \subseteq \Sphere^{n-1}$ be a spherical design of harmonic index $t$.
If we put $c_{n,t}=-\min Q_{n,t}(x)$, then the following inequality holds:
\begin{equation}\label{equ:4.1}
\left| Y \right| \geq b_{n,t} := 1+\frac{Q_{n,t}(1)}{c_{n,t}}.
\end{equation}
Moreover, the equality holds in $\eqref{equ:4.1}$ if and only if $Q_{n,t}(\alpha) = -c_{n,t}$ for any $\alpha \in I(Y)$.
\end{theorem}

In the view point of Section~\ref{sec:LP}, Theorem~\ref{theo:2.1} is the result from putting
\begin{equation}\label{eq:cnt}
F(x) = c_{n,t} + Q_{n,t}(x) \qquad (c_{n,t} := -\min Q_{n,t}(x)),
\end{equation}
and the inequality~\eqref{eq:mie}.
Recall from Proposition~\ref{ourtestfunction} that if $Y\subseteq \Sphere^{n-1}$ is a tight spherical design of harmonic index $t$, then $I(Y)$ is contained in the set of roots of $F(x)$.

Delsarte--Goethals--Seidel \cite[Theorem 4.8]{DGS1977} (1977) gave an upper bound for a spherical $s$-distance set $X\subseteq \Sphere^{n-1}$, namely, $\left| X \right| \leq \binom{n+s-1}{n-1}+\binom{n+s-2}{n-1}$.
In particular, for a spherical $2$-distance set $X \subseteq \Sphere^{n-1}$ with $I(X) = \{ \alpha, \beta \}$ and $\alpha+\beta \geq 0$, a better upper bound $\frac{n(n+1)}{2}$ for $\left| X \right|$ is given by Musin \cite[Theorem 1]{M2009} (2009).

\subsection{The non-existence of tight spherical designs of harmonic index $6$}
In this subsection $Y \subseteq \Sphere^{n-1}$ denotes a spherical design of harmonic index $6$.
The Gegenbauer polynomial $Q_{n,6}(x)$ is
\[\begin{array}{rcl}
Q_{n,6}(x) & = & {\scriptstyle \frac{n(n+2)(n+10)}{6!} \left\{ (n+4)(n+6)(n+8)x^6 -15(n+4)(n+6)x^4 + 45(n+4)x^2-15 \right\} }.\\
\end{array} \]
By taking the largest root for $Q_{n,6}^{\prime}(x)=0$, we get (see e.g. the proof of \cite[Corollary 4.1]{BOT2015}) the point $\alpha$ at which $Q_{n,6}(x)$ takes the minimum value, i.e., $Q_{n,6}(\alpha) = -c_{n,6}$.
The lower bound $b_{n,6}$ can be obtained as well.
The following are our calculation results:
\[ \begin{array}{l}
\alpha^2=\frac{5(n+6)+\sqrt{10(n+3)(n+6)}}{(n+6)(n+8)}, \\
c_{n,6}=-\frac{n(n+2)(n+10) \left( 2(n-2)(n+3)(n+6)+(n+3)(n+4)\sqrt{10(n+3)(n+6)}\right)}{36(n+6)(n+8)^2}, \\
b_{n,6}=\frac{(n+4)\left( 20\sqrt{10(n+3)(n+6)}+(n+3)(n+6)(n^2+9n-12) \right)}{20\left( 2(n-2)(n+6)+(n+4)\sqrt{10(n+3)(n+6)} \right)}.
\end{array}\]

It is not difficult to check that $|Y|=b_{n,6}> \frac{n(n+1)}{2}$ if $n\geq 37$.
Moreover, $b_{2,6}=2$ and $b_{24,6}=231$ are the only two cases for which $b_{n,6} \in \mathbb{Z}$ when $n \leq 36$.

\begin{remark}\label{remark:4.2}
We should remark that not all the roots of $F(x)$ in \eqref{eq:cnt} will necessarily appear in $I(Y)$ when $n$ is small.
Consider the case $n=2$.
Recall that $b_{2,2e} = 2$ is proved for general $e$ in \cite[p. 6]{BOT2015}.
Let $y_1,y_2$ be two unit vectors in $\Real^2$ with angular $\theta =j \pi/{2e}$ for odd $j$.
Then, by the argument in \cite[p. 2]{BOT2015}, $Y=\{ y_1, y_2 \}$ is a tight spherical design of harmonic index $2e$ on $\Sphere^1$.
\end{remark}

Larman--Rogers--Seidel (1977) proved the following fact.
\begin{theorem}[{\cite[Theorem 2]{LRS1977}}]\label{theo:4.5}
Let $X$ be a $2$-distance set in $\mathbb{R}^n$ with Euclidean distances $c$ and $d$ $(c < d)$.
If $|X|>2n+3$, then we have
\[ \frac{c^2}{d^2} = \frac{(k-1)}{k} \]
for some integer $k$ with $2 \leq k \leq \frac{1+\sqrt{2n}}{2}$.
\end{theorem}

Suppose $n=24$.
Then $Y$ is a spherical $2$-distance set in $\Sphere^{23}$ with $I(Y) = \{ \pm \alpha\}$.
If we put $c=\sqrt{2-2\alpha}$ and $d=\sqrt{2+2\alpha}$, then $c$ and $d$ become the Euclidean distances between two distinct vectors in $Y$.
However, in this case, we obtain $c^2/d^2=1/3$ from easy calculation, contrary to Theorem~\ref{theo:4.5}.
(Note that $b_{24,6} = 231 > 2 \times 24+3 = 51$.)
Hence there exists no tight spherical design of harmonic index $6$ when $n=24$.

\begin{remark}
If $Y$ is a tight spherical design of harmonic index $T$, then $I(Y)$ is contained in the set of roots of $F(x)$.
Suppose $F(x)$ has $s$ roots, i.e., $\left| I(Y) \right| \leq s$, then $Y$ is a spherical $s^{\prime}$-distance set with $s^{\prime} \leq s$.
It is known that the upper bound for spherical $s^{\prime}$-distance set is $A(n,s^{\prime}):=\binom{n+s^{\prime}-1}{n-1}+\binom{n+s^{\prime}-2}{n-1}$. Since $\left| Y \right|=b_{n,T} \leq A(n,s^{\prime})$ and $A(n,s^{\prime}) \leq A(n,s)$, we get a less restrictive condition on $n$ if $s^{\prime}=s$. This implies that it is easier to prove the non-existence of $s^{\prime}$-distance set with $s^{\prime} < s$. Therefore, in what follows, a tight spherical design of harmonic index $T$ is equivalent to a spherical $s$-distance set.
\end{remark}
\subsection{The non-existence of tight spherical designs of harmonic index $8$}
In this subsection $Y \subseteq \Sphere^{n-1}$ denotes a spherical design of harmonic index $8$.
The Gegenbauer polynomial $Q_{n,8}(x)$ is
\[ \begin{array}{rcl}
Q_{n,8}(x)&=& {\scriptstyle \frac{n(n+2)(n+4)(n+14)}{8!} \big\{ (n+6)(n+8)(n+10)(n+12)x^8 -28(n+6)(n+8)(n+10)x^6} \\
& & \quad\qquad\qquad\qquad {\scriptstyle +210(n+6)(n+8)x^4 - 420(n+6)x^2+105\big\} }.\\
\end{array} \]
As in the preceding subsection we can obtain $\alpha$, $c_{n,8}$, and also
\begin{equation}\label{eq:asymbn8}
b_{n,8}=\frac{1}{252 \times 12.03144913\cdots}\times n^4(1+o(1)).
\end{equation}
It can be checked that $\left| Y \right| = b_{n,8} > \frac{n(n+1)}{2}$ if $n \geq 20$ and, if $n \leq 19$,  the only integral value is $b_{2,8}=2$.
By a similar argument as in Remark~\ref{remark:4.2} one trivial example exists when $n=2$.
\begin{remark}
We do not give the formulas of $\alpha$, $b_{n,8}$ and $c_{n,8}$ explicitly, since they are extremely complicated.
Here, $b_{n,8} > \frac{n(n+1)}{2}$ is checked from the formula of $b_{n,8}$ rather than from the asymptotic form \eqref{eq:asymbn8}.
\end{remark}

\section{The asymptotic non-existence of tight spherical designs of harmonic index $2e$ for general $e$}
\label{sec:asymptotic}

\noindent
In this section we consider the existence of tight spherical designs of harmonic index $2e$ for $e\geq 5$, since the cases $e=2,3,4$ were already treated.
Our main result in this section is the following theorem.

\begin{theorem}\label{theo:5.1}
Let $e \geq 2$ be fixed.
Then there exist positive constants $A_{2e}$ and $B_{2e}$ such that
\[\lim_{n\rightarrow \infty}\frac{c_{n, 2e}}{n^e}=A_{2e} \quad\mbox{ and }\quad \ \lim_{n\rightarrow \infty}\frac{b_{n, 2e}}{n^e}=B_{2e},\]
where $A_{2e}$ and $B_{2e}$ depend only on $e$.
Therefore,
\[c_{n,2e}=A_{2e}n^e(1+o(1)) \quad\mbox{ and }\quad \ b_{n,2e}=B_{2e}n^e(1+o(1)).\]
\end{theorem}

\begin{corollary}
Let $e \geq 3$ be fixed.
If $n$ is sufficiently large, then there exist no tight spherical designs of harmonic index $2e$.
\end{corollary}

\begin{proof}
If $Y$ is a tight spherical design of harmonic index $2e$, then $I(Y) \subseteq \{\pm \alpha\}$ for some $\alpha>0$, and so $\left| Y \right| \leq \frac{n(n+1)}{2}$.
On the other hands, if $n$ is sufficiently large, then Theorem~\ref{theo:5.1} implies
\[ \left| Y \right| = b_{n,2e}=B_{2e}n^e(1+o(1)),\]
a contradiction.
\end{proof}

\begin{proof}[Proof of Theorem~\ref{theo:5.1}]
Szeg\"{o} \cite[p. 107]{S2003} gave the asymptotic property of Gengebauer polynomial $C_t^{\lambda}(x)$:
\[\lim_{\lambda\rightarrow \infty} \lambda^{-\frac{t}{2}}C_t^{\lambda}(\lambda^{-\frac{1}{2}} x) = \frac{H_{t}(x)}{t!},\]
where $H_t(x)$ is the Hermite polynomial of degree $t$.\\
Recall that if $n\geq 3$ then $Q_{n,t}(x)=\frac{n+2t-2}{n-2}C_t^{(n-2)/2}(x)$.
(See e.g. \cite[p. 365]{DGS1977}.)
Putting $\lambda=\frac{n-2}{2}$ and $t = 2e$, we have
\[ \begin{array}{rcl}
{\displaystyle \lim_{\lambda\rightarrow \infty} \lambda^{-\frac{t}{2}}C_t^{\lambda}(\lambda^{-\frac{1}{2}}x)} & = & {\displaystyle \lim_{n\rightarrow \infty} \left( \frac{n-2}{2}\right)^{-e}\frac{n-2}{n+4e-2} \, Q_{n,2e}(\sqrt{\frac{2}{n-2}}\, x ) } \\
 & & \\
 & = & {\displaystyle 2^e\lim_{n\rightarrow \infty} n^{-e}Q_{n,2e}(\sqrt{\frac{2}{n-2}}\, x)}.
\end{array} \]
Set $P_{n,e}(x)=n^{-e}Q_{n,2e}(\sqrt{\frac{2}{n-2}}\, x)$ for simplicity.
Then we have
\begin{equation}\label{equ:5.1}
2^e\lim_{n\rightarrow \infty} P_{n,e}(x)=\frac{H_{2e}(x)}{(2e)!}.
\end{equation}
Take the derivative with respect to $x$ on both sides of (\ref{equ:5.1}).
Since $P^{\prime}_{n,e}(x)$ uniformly converges to $\frac{2^e}{(2e-1)!}x^{2e-1}$ for fixed $e$, we get the following result:
\[
2^e \lim_{n\rightarrow \infty} \frac{d}{dx} P_{n,e}(x)=2^e\frac{d}{dx} \left( {\lim_{n\rightarrow \infty} P_{n,e}(x)} \right)=\frac{H^{\prime}_{2e}(x)}{(2e)!}=\frac{4e}{(2e)!}{H_{2e-1}(x)},
\]
where the last equality is due to the property $H_t^{\prime}(x)=2tH_{t-1}(x)$.
Let $x_1$ be the largest zero of $H_{2e-1}(x)$.
Then
\[2^e \lim_{n\rightarrow \infty}  \frac{d}{dx}\left( n^{-e}Q_{n,2e}(\sqrt{\frac{2}{n-2}}\, x)\right) \Big|_{x=x_1}=\frac{1}{(2e)!}H^{\prime}_{2e}(x_1)=0.\]
Thus the following equality can be obtained.
\[A_{2e}=-\lim_{n\rightarrow \infty}\frac{\min Q_{n, 2e}(x)}{n^e}=-\lim_{n\rightarrow \infty}\frac{ Q_{n, 2e}(\sqrt{\frac{2}{n-2}}\, x_1)}{n^e}=-\frac{H_{2e}(x_1)}{2^e (2e)!}.\]
Recall that $b_{n,t}=1+\frac{Q_{n,t}(1)}{c_{n,t}}$ and $Q_{n,t}(1) = \binom{n+t-1}{n-1} - \binom{n+t-3}{n-1}$. This implies
\[B_{2e}=\lim_{n\rightarrow \infty}\frac{b_{n, 2e}}{n^e}=\frac{1}{(2e)!A_{2e}}=-\frac{2^e}{H_{2e}(x_1)}.\]
\end{proof}

\begin{remark}
In Theorem~\ref{theo:5.1} we did not give explicit evaluation of $B_{2e}$, but it is possible to give it, since the locations of the zeros of Hermite polynomials and the (local) minimum values of $H_{2e}(x)$ are well studied.
Also, if we want to evaluate $b_{n,2e}$ explicitly from below, rather than evaluating $B_{2e}$, it is also possible, although we will not discuss it in this paper.
For this purpose, the following papers \cite{BDHSS2010}, \cite{DN2010}, \cite{EL1992}, \cite{GL2003}  may be useful to do that.
(It seems there are many literature on this.)
\end{remark}

\section{Tight spherical designs of harmonic index $\{8,4\}$}
\label{sec:84}

\noindent
Consider the case of $T=\{t_1 (=\! 2e),\, t_2\}$ with $t_2=$ even and $t_1 > t_2$.
With the argument in Section~\ref{sec:Fisherandtight}, we are interested in the cases when $Q_{n,2e}(x)+f_{t_2}Q_{n, t_2}(x)$ has the minimum value $-c_{n,T}$ at only two non-negative points $\alpha, \beta$ (with $\alpha>\beta$). Namely, by taking
\[L(x)=Q_{n,2e}(x)+f_{t_2}Q_{n, t_2}(x),\]
we want to determine $f_{t_2}$ such that
\[ Q_{n,2e}(x)+f_{t_2}Q_{n,t_2}(x)=a (x^2-\alpha^2)^2 (x^2-\beta^2)^2-c_{n,T} \]
for some $\alpha, \beta$ and $c_{n,T}$.

\noindent
Suppose $t_1=8$ and $t_2=4$.
Then the problem is to find $f_4$ such that
\begin{equation}\label{equ:6.1}
Q_{n,8}(x)+f_4Q_{n,4}(x)=a(x^2-\alpha^2)^2(x^2-\beta^2)^2-c_{n,T}.
\end{equation}
The Gegenbauer polynomial $Q_{n,4}(x)$ is
\[Q_{n,4}(x)=\frac{n(n+6)}{4!} \left\{ (n+2)(n+4)x^4-6(n+2)x^2+3 \right\}.\]
By comparing the coefficients in \eqref{equ:6.1}, we obtain the following equations:
\[
\begin{array}{l}
{\scriptsize a=\frac{n(n+2)(n+4)(n+6)(n+8)(n+10)(n+12)(n+14)}{8!}, }\\
\\
{\scriptsize 2a(\alpha^2+\beta^2)=\frac{28n(n+2)(n+4)(n+6)(n+8)(n+10)(n+14)}{8!}, }\\
 \\
{\scriptsize  a(\alpha^4+\beta^4+4\alpha^2\beta^2)=\frac{210n(n+2)(n+4)(n+6)(n+8)(n+14)}{8!} +\frac{n(n+2)(n+4)(n+6)}{4!}f_4,}\\
\\
{\scriptsize 2a\alpha^2\beta^2(\alpha^2+\beta^2)=\frac{420n(n+2)(n+4)(n+6)(n+14)}{8!}+\frac{6n(n+2)(n+6)}{4!}f_4, }\\
\\
{\scriptsize a\alpha^4\beta^4-c_{n,T}=\frac{105n(n+2)(n+4)(n+14)}{8!}+\frac{3n(n+6)}{4!}f_4. }
\end{array}
\]
Therefore,
\[ \begin{array}{l}
\ds \alpha^2+\beta^2=\frac{14}{n+12}, \\
 \\
\scriptsize (\alpha^2+\beta^2)^2+2\alpha^2\beta^2=\frac{210}{(n+10)(n+12)}+\frac{1680}{(n+8)(n+10)(n+12)(n+14)}f_4,\\
\\
\scriptsize \alpha^2\beta^2(\alpha^2+\beta^2)=\frac{210}{(n+8)(n+10)(n+12)}+\frac{5040}{(n+4)(n+8)(n+10)(n+12)(n+14)}f_4.
\end{array} \]
We have
\[ f_4=\frac{(n+4)(n+5)(n+14)}{60(n+12)}.\]
Hence $\alpha^2$ and $\beta^2$ are the solutions of the following quadratic equation in the variable $u$:
\[ u^2-\frac{14}{n+12}u +\frac{21}{(n+8)(n+12)}=0.\]
Finally, we obtain
\[
\begin{array}{l}
\ds \alpha^2,\beta^2=\frac{7(n+8) \pm 2\sqrt{7(n+5)(n+8)}}{(n+8)(n+12)},\\
\\
\ds c_{n,T}=\frac{n(n+1)(n+4)(n+5)(n+10)(n+14)}{160(n+8)(n+12)},\\
\\
\ds b_{n,T}=\frac{1}{252}(n+1)(n+2)(n+5)(n+6).
\end{array}
\]
If a tight spherical design of harmonic index $\{8,4\}$ exists, then it is a spherical $4$-distance set $\{\pm \alpha, \pm \beta\}$.
We define
\[U(h) := \left\lfloor \; \frac{1}{2} + \sqrt{\frac{h^2}{2h-2}+\frac{1}{4}} \;\right\rfloor.\]
For a spherical $s$-distance set, Nozaki (2011) generalized Larman--Rogers--Seidel theorem \cite[Theorem 2]{LRS1977} as follows.
\begin{theorem}[{\cite[Theorem 5.1]{N2011}}]\label{theo:6.1}
Let $Y$ be an $s$-distance set on $\Sphere^{n-1}$ with $s \geq 2$ and $I(Y) = \{ \beta_1, \ldots, \beta_s \}$.
Put $N := \binom{n+s-2}{s-1} + \binom{n+s-3}{s-2}$.
If $\left| Y \right| \geq 2N$, then for each $i = 1,\dots,s$,
\[ k_i := \prod_{j=1,\dots,s, \atop j \neq i} \frac{1 - \beta_j}{\beta_i - \beta_j} \]
must be an integer with $\left| k_i \right| \leq U(N)$.
\end{theorem}

If $X=Y \cup (-Y)$ is an antipodal spherical $s$-distance set, then $Y$ is a spherical $(s-1)$-distance set.
Nozaki (2011) proved the following theorem.
(The conditions (lower bounds) of $\left| X \right|$ in Theorem~\ref{theo:6.2} are less restrictive than that in Theorem~\ref{theo:6.1}.)

\begin{theorem}[{\cite[Theorem 5.2]{N2011}}]\label{theo:6.2}
Let $X$ be an antipodal $s$-distance set on $\Sphere^{n-1}$ where $s$ is an odd integer at least $5$. \\
Suppose $I(X)=\{-1,\pm\beta_1, \pm\beta_2, \ldots, \pm\beta_{\frac{s-1}{2}} \}$.
\begin{enumerate}[$(1)$]
\item Let $N=\binom{n+s-4}{s-3}$. If $\left| X \right| \geq 4N$, then for each $i = 1,\dots,(s-1)/2$,
\[ k_i := \prod_{j=1,\dots,\frac{s-1}{2}, \atop j \neq i} \frac{1 - \beta_j^2}{\beta_i^2 - \beta_j^2} \]
must be an integer with $\left| k_i \right| \leq U(N)$.
\item Let $N=\binom{n+s-3}{s-2}$. If $\left| X \right|\geq 4N+2$, then for each $i = 1,\dots,(s-1)/2$,
\[ k_i := \frac{1}{\beta_i} \prod_{j=1,\dots,\frac{s-1}{2}, \atop j \neq i} \frac{1 - \beta_j^2}{\beta_i^2 - \beta_j^2} \]
must be an integer with $\left| k_i \right| \leq \lfloor \sqrt{2N^2/(N+1)} \rfloor$.
\end{enumerate}
\end{theorem}

With the above theorem and {\cite[Theorem 5.3]{N2011}}, Nozaki showed that inner products for an antipodal spherical $s$-distance set are in fact rational.

\begin{theorem}[{\cite[Theorem 5.4]{N2011}}]\label{theo:6.3}
Suppose $X$ is an antipodal $s$-distance set in $\Sphere^{n-1}$ with $s \geq 4$.
If $|X| \geq 4\binom{n+s-3}{s-2}+2$, then $\beta$ is rational for
any $\beta \in I(X)$.
\end{theorem}

A tight spherical design of harmonic index $\{8,4\}$ is regarded as a spherical $4$-distance set $Y$ with $I(Y)=\{\pm\alpha, \pm\beta\}$.
We construct an antipodal spherical $5$-distance set $X^{\prime}=Y \cup (-Y)$.
(Note that $I(X^{\prime})=\{-1,\pm\alpha, \pm\beta\}$.)
By applying Theorem~\ref{theo:6.2} on the set $X^{\prime}$ for $s=5$, we obtain the next lemma.
\begin{lemma}\label{lem:rat}
Suppose $Y$ is a spherical $4$-distance set $\{\pm \alpha, \pm \beta\}$. Let $X^{\prime}=Y \cup (-Y)$.
\begin{enumerate}[$(1)$]
\item If $ \left| Y \right| \geq 2\binom{n+1}{2}$, then the following two numbers are integers:
\[ k_1=\frac{1-\alpha^2}{\beta^2-\alpha^2}, \qquad k_2=\frac{1-\beta^2}{\alpha^2-\beta^2}. \]
\item If $\left| Y \right| \geq 2\binom{n+2}{3}+1$, then the following two numbers are integers:
\[ k_1=\frac{1-\alpha^2}{\beta(\beta^2-\alpha^2)}, \qquad k_2=\frac{1-\beta^2}{\alpha(\alpha^2-\beta^2)}. \]
\end{enumerate}
\end{lemma}
\begin{theorem}
There exists no tight spherical design of harmonic index $\{8, 4\}$ on $S^{n-1}$ for all $n$.
\end{theorem}
\begin{proof}
If $Y$ is a tight spherical design of harmonic index $\{8,4\}$, then
\[ \left| Y \right|=\frac{(n+1)(n+2)(n+5)(n+6)}{252} \mbox{ with } I(Y)=\{ \pm \alpha, \pm \beta \},\]
where
\[\alpha,\beta=\sqrt{\frac{7(n+8) \pm 2\sqrt{7(n+5)(n+8)}}{(n+8)(n+12)}}.\]
We shall consider three cases: (1) $n\geq 76$, (2) $9\leq n \leq 75$, and (3) $2 \leq n \leq 8$.

\noindent \textbf{Case (1)}:
If $n \geq 76$, then
\[ {\textstyle \frac{(n+1)(n+2)(n+5)(n+6)}{252}  \geq  2\left( \binom{n+2}{3}+\binom{n+1}{2} \right) > 2\binom{n+2}{3}+1 }. \]
By Theorem~\ref{theo:6.1}, $k_1$, $k_2$ are integers.
We have
\[k_2=\frac{1-\beta^2}{\alpha^2-\beta^2}= \frac{2+\sqrt{\frac{(n+5)(n+8)}{7}}}{4}=z \in \mathbb{Z}.\]
Hence $(n+5)(n+8)=7(4z-2)^2$.
By Lemma~\ref{lem:rat} we have
\[\alpha\beta = \sqrt{\frac {21} {(n+8)(n+12)}} \in \Rational.\]
Then $(n+8)(n+12)=21p^2/q^2$ for some coprime integers $p$ and $q$.
Furthermore $21p^2/q^2$ should be integer.
Thus $q^2|21$, i.e., $q=1$.
We have $(n+8)(n+12)=21p^2$ and get the following table for some integers $y_1,y_2,y_3$.
\[ \begin{array}{c|c|c|c}
&n+5 & n+8 & n+12\\ \hline
\text{(i)} & y_1^2 & 7y_2^2 & 3y_3^2\\
\text{(ii)}& 7y_1^2 & y_2^2 & 21y_3^2\\
\text{(iii)} &3y_1^2 & 21y_2^2  & y_3^2\\
\text{(iv)} & 21y_1^2 & 3y_2^2 & 7y_3^2 \\ \hline
\end{array}\]
We know that $\gcd(n+5,n+8)=1$ or $3$.
If $\gcd(n+5,n+8)=1$, then $(n+5)(n+8)=7(4z-2)^2$ implies that $n+5=y_1^2$, $n+8=7y_2^2$ or $n+5=7y_1^2$, $n+8=y^2_2$.
If $\gcd(n+5,n+8)=3$, then $n+5=3y_1^2$, $n+8=21y_2^2$ or $n+5=21y_1^2$, $n+8=3y^2_2$.\\
For case (i), $n+12=3y_3^2$ is obtained from $(n+8)(n+12)=21p^2$. We can similarly get the other three cases in the above table.

Then we can prove the non-existence.
\begin{enumerate}[(i)]
\item $7=3y_3^2-y_1^2$ implies $y_1^2 \equiv 2 \; (\mathrm{mod} 3)$. Impossible.
\item $7=21y_3^2-7y_1^2$ implies $y_1^2 \equiv 2 \; (\mathrm{mod} 3)$. Impossible.
\item $3=21y_2^2-3y_1^2$ implies $y_1^2 \equiv 6 \; (\mathrm{mod} 7)$. Impossible.
\item  We can not get contradiction from basic observation, but this problem can be formulated as the integral solutions of the following equation:
\[y^2=(n+5)(n+8)(n+12).\]
 By linear transformation $x=n+8$, this equation becomes $y^2=x^3+x^2-12x$.
From the database of elliptic curve with LMFDB label 168.b2, we know that $(x,y)=(-4,0),(0,0),(3,0)$ are all integral solutions of $y^2=x^3+x^2-12x$, namely, $y^2=(n+5)(n+8)(n+12)$ has no non-trivial integral solution.
\end{enumerate}
\begin{remark}
(1)
Any elliptic curve over $\Rational$ has a Weierstrass model (or equation) of the form
\begin{equation}\label{equ:dstar} \tag{$\ast\ast$}
 y^2+a_1 x y+a_3 y = x^3+a_2 x^2 +a_4 x+a_6.
 \end{equation}
They are often displayed as a list $[a_1,a_2,a_3,a_4,a_6]$.
More information about the database of elliptic curve is available from:

\verb@http://www.lmfdb.org/EllipticCurve/Q@

\noindent(2)
The integral solutions of some elliptic equations of form \eqref{equ:dstar} can be solved using SAGE \cite{SAGE} with following two commands (the reader should put suitable values for $a_1,a_2,a_3,a_4,a_6$):

\verb@E=EllipticCurve(QQ,[@$a_1,a_2,a_3,a_4,a_6$\verb@])@

\verb@E.integral_points()@
\end{remark}

\noindent \textbf{Case (2)}:
If $9 \leq n \leq 75$, then
\[ \frac{(n+1)(n+2)(n+5)(n+6)}{252}  \geq  2\binom{n+1}{2}. \]
Using the first statement in Lemma~\ref{lem:rat}, we see that both $\frac{1-\alpha^2}{\beta^2-\alpha^2}$ and $\frac{1-\beta^2}{\alpha^2-\beta^2}$ are integers.
It is easy to check that neither of them is integer for $9 \leq n \leq 75$.

\noindent \textbf{Case (3)}:
If $2 \leq n \leq 8$, $b_{8,T}=65$ is the unique case with $b_{n,T} \in \Integer$. We set up the semidefinite programming (SDP) method on the upper bounds for spherical 4-distance sets with the indicated inner product values.
The SDP formula can be obtained from special setting of Bachoc-Vallentin \cite[p. 10--11]{BV2008} or generalization of Barg-Yu \cite[Theorem 3.1]{BY2013} for spherical $2$-distance sets.
We choose the positive semidefinite  matrices $S^n_k$  with size $(9-k) \times (9-k)$ %$k=1,2,3,4$ and $5$
 and linear constraints $\sum_{c_i, c_j \in X} G^n_k ( \langle c_i, c_j \rangle) \geq 0 $ for $k=1$. %$k=1, \cdots , 8$.
($S^n_k$ and $G^n_k$ are the same notation in \cite{BY2013}).
When $n=8$, the SDP upper bound for the spherical 4-distance set is $50.23$ which is strictly less than our LP lower bound $b_{8,T}=65$. We can conclude there exists no tight spherical design of harmonic index $\{8,4\}$.
\end{proof}

\section{Tight spherical designs of harmonic index $\{6,4\}$, $\{6,2\}$, $\{8,6\}$, $\{8,2\}$, as well as $\{10,6,2\}$, $\{12,8,4\}$}
\label{sec:several}
\noindent
We are interested in the cases when $T=\{t_1,t_2,\ldots,t_{\ell}\}$ and $L(x)$ takes the minimum value $-c_{n,T}$ at $\ell$ non-negative points, where \[ L(x) = Q_{n,2e}(x) + f_{t_2}Q_{n,t_2}(x) + \cdots + f_{t_{\ell}}Q_{n,t_{\ell}}(x). \]
In this section, we will prove the non-existence of tight spherical designs of harmonic index $T$ for some $T$ with $\ell=2$ or $3$.

\subsection{Non-existence of tight spherical harmonic designs of index $\{6,4\}$}
Find $f_4$ such that
 \begin{equation}\label{equ:7.3}
Q_{n,6}(x)+f_4Q_{n,4}(x)=ax^2(x^2-\alpha^2)^2-c_{n,T}.
\end{equation}
By comparing the coefficients in \eqref{equ:7.3}, we get the following results:
\[
\begin{array}{l}
 \ds \alpha^2=\frac{15}{2(n+8)}-\frac{15}{(n+8)(n+10)}f_4,\\
\\
\ds \alpha^4=\frac{45}{(n+6)(n+8)}-\frac{180}{(n+4)(n+8)(n+10)}f_4.
\end{array}
\]
Solving $f_4$ and $\alpha^2$ from these two equations gives:
\[ \begin{array}{l}
f_4= \dfrac{n+10}{10(n+4)(n+6)} \left( (n+6)(n-12) \pm 2(n+8)\sqrt{-(n+6)(n-4)} \right) , \\
\\
\alpha^2 = \dfrac{3\left( 2(n+6) \pm \sqrt{-(n+6)(n-4)} \right)}{(n+4)(n+6)}.
\end{array} \]
If $n \geq 5$, then $f_4$ is a complex number.
So we can not find $L(x)=Q_{n,6}(x)+f_4Q_{n,4}(x)$ satisfying our assumption.
It is easy to check that $b_{2,T}=2$, $b_{3,T}=3$ (or 0) and $b_{4,T}=2$.
By Remark~\ref{remark:4.2}, there exists no tight spherical design of harmonic index $\{6,4\}$ when $n=2$ and $\left| Y \right|=2$.
When $n=3$ and $n=4$, observe that the lower bounds for spherical design of harmonic index $6$  are about $3.41$ and $5.29$, respectively, which are strictly larger than $b_{3,T}$ and $b_{4,T}$.
(Note that spherical design of harmonic index $\{6,4\}$ should also satisfy the condition for harmonic index $6$.)
From the discussion above, there exists no tight spherical design of harmonic index $\{6,4\}$ for any $n$.

\subsection{Non-existence of tight spherical harmonic designs of index $\{6,2\}$}
The Gegenbauer polynomial $Q_{n,2}(x)$ is
\[ Q_{n,2}(x)=\frac{(n+2)(nx^2-1)}{2}. \]
Let $Q_{n,6}(x)+f_2Q_{n,2}(x)=ax^2(x^2-\alpha^2)^2-c_{n,T}$.
With similar calculation we have the following results.
\[ \begin{array}{l}
{\displaystyle f_2=\frac{(n-2)(n+4)(n+10)}{32(n+8)}, \quad \alpha=\pm \sqrt{\frac{15}{2(n+8)}}}, \\
\\
{\displaystyle c_{n,T}=\frac{(n+2)(n+6)(n+10)(7n-4)}{192(n+8)}},\\
\\
{\displaystyle b_{n,T}=\frac{n(n+4)(2n+1)^2}{15(7n-4)}}, \\
\\
{\displaystyle \quad \quad = \frac{1}{2401 \times 15}(1372n^3+7644n^2+10199n+7200)+ \frac{1920}{2401(7n-4)}}.
\end{array} \]
\noindent \textbf{Case (1)}:
 If $\left| \frac{1920}{2401(7n-4)}\right|< \frac{1}{2401 \times 15}$, i.e., $n \geq 8817$, then  $b_{n,T}$ is not integer.

\noindent \textbf{Case (2)}: Tight spherical design $Y$ of harmonic index $\{6,2\}$ is regarded as a spherical $3$-distance set with $I(Y) = \{0,\pm \alpha\}$.\\
Lemmens-Seidel (1973) proved the following fact.
\begin{theorem}[{\cite[Theorem 3.4]{LS1973}}]
If there are $|X|$ equiangular lines with angle $\arccos$ $\alpha$ in Euclidean $n$-dimensional space $\mathbb{R}^n$, and if $|X| > 2n$, then $1/ \alpha$ is an odd integer.
\end{theorem}
Then we can give a weaker condition for $\alpha$ with $3$-distance set as follows.
\begin{theorem}\label{theo:7.2}
If $Y \subseteq \Sphere^{n-1}$ is a spherical $3$-distance set with $I(Y) = \{0, \pm \alpha\}$ and $\left| Y \right| > 2n$, then $1/ \alpha$ is an integer.
\end{theorem}

\begin{proof}
Let $Y$ be a set of unit vectors whose mutual inner products is $\{0, \pm \alpha\}$, and let $G$ be the Gram matrix of such vectors.
Then,
\begin{equation*}
G=\left(\begin{array}{cccc}
1 & & x \\
& \ddots &\\
x & & 1
\end{array}\right), \quad
A=\frac{1}{\alpha}(G-I)=
\left(\begin{array}{cccc}
0 & & x^{\prime} \\
& \ddots &\\
x^{\prime} & & 0
\end{array}\right).
\end{equation*}
where $x \in \{0, \pm \alpha\}$ and $x^{\prime} \in \{0, \pm 1 \}$.

$G$ is a symmetric and positive semi-definite matrix of order $|Y|$. It has the smallest eigenvalue $0$ of multiplicity $m \geq |Y|-n$.
Therefore, $A$ has the smallest eigenvalue $-1/ \alpha$ of multiplicity $m \geq |Y|-n$.
Moreover, $-1/ \alpha$ is an algebraic integer since $A$ is an integer matrix, and every algebraic conjugate of $-1/ \alpha$ is also an eigenvalue of $A$ with multiplicity $m$.
If $|Y| > 2n$, then $m > \frac{|Y|}{2}$.
Note $A$ can not have more than one eigenvalue of multiplicity $m$ because $A$ is a $|Y| \times |Y|$ matrix.
Therefore $-1/ \alpha$ is rational, since it is also an algebraic integer, hence $-1/ \alpha$ is an integer.
\end{proof}
If $5 \leq n \leq 8816$, then $\left| Y \right| >2n$. By Theorem~\ref{theo:7.2}, we know that $1/ \alpha \in \Integer$. And it is easy to check that $\frac{1}{\alpha}=\sqrt{\frac{2(n+8)}{15}} \in \Integer$ and $b_{n,T} \in \Integer$ can not hold simultaneously for $5 \leq n \leq 8816$.

\noindent \textbf{Case (3)}:
 If $2 \leq n \leq 4$, then $b_{2,T}=2$ is the unique integral case. By Remark~\ref{remark:4.2}, $Y=\{(1,0), (0,1)\}$ is a tight spherical design of harmonic index $\{6,2\}$ in $\Sphere^1 \subseteq \Real^2$.

\subsection{Non-existence of tight spherical harmonic designs of index $\{8,6\}$}
 Find $f_6$ such that
\begin{equation}\label{equ:7.4}
Q_{n,8}(x)+f_6Q_{n,6}(x)=a(x^2-\alpha^2)^2(x^2-\beta^2)^2-c_{n,T}.
\end{equation}

By comparing the coefficients in \eqref{equ:7.4}, we get the following results:

\[ \begin{array}{l}
\ds \alpha^2+\beta^2=\frac{14}{n+12}-\frac{28}{(n+12)(n+14)}f_6, \\
 \\
\ds (\alpha^2+\beta^2)^2+2\alpha^2\beta^2=\frac{210}{(n+10)(n+12)}-\frac{840}{(n+8)(n+12)(n+14)}f_6, \\
\ds \alpha^2\beta^2(\alpha^2+\beta^2)=\frac{210}{(n+8)(n+10)(n+12)}-\frac{1260}{(n+6)(n+8)(n+12)(n+14)}f_6.
\end{array} \]

Set $Z_1=\alpha^2+\beta^2$, $Z_2=\alpha^2 \beta^2$.
Then we have the following relations:
\[
{\displaystyle f_6=\frac{n+14}{2}\left(1-\frac{n+12}{14}Z_1 \right)}, \quad
{\displaystyle Z_2=\frac{15 \big( 3(n+10)Z_1-28 \big)}{(n+6)(n+8)(n+10)Z_1}}.
\]
We can obtain that
\[Z_1=\alpha^2+\beta^2=\frac{2(10g)^{\frac{1}{3}}}{(n+6)(n+8)(n+10)}+\frac{40(n+3)}{(n+8)(10g)^{\frac{1}{3}}}+\frac{10}{n+8}.\]
where
\[ {\textstyle g=\left(-(n-2)(n+3)+(n+3)(n+8)\sqrt{\frac{n(n-4)}{(n+6)(n+10)}}\right)(n+6)^2(n+10)^2 }.\]
In this case, the lower bound for $\left| Y \right|$ is
\[ b_{n,T}=\frac{1}{252 \times 9.427094401\cdots} \times n^4(1+o(1)).\]

\noindent \textbf{Case (1)}:
 If $n \geq 759$, then $\left| Y \right| \geq 2 \binom{n+2}{3}+1$. Using Theorem~\ref{theo:6.3}, we know that $\alpha, \beta \in \Rational$.
So one necessary condition is $\sqrt{\frac{n(n-4)}{(n+6)(n+10)}} \in \Rational$.
Equivalently, $n(n-4)(n+6)(n+10)=(n^2+6n-20)^2-400=u^2$ for some integer $u$, i.e.,
\[(n^2+6n-20+u)(n^2+6n-20-u)=400.\]
Moreover, these two factors of $400$ have the same parity, since the difference is $2u$.
Then they can be one of the following cases, $2 \times 200$, $4 \times 100$, $8 \times 50$, $10 \times 40$ and $20 \times 20$.
But this is satisfied only when $n=4,6$.

\noindent \textbf{Case (2)}:
If $2 \leq n \leq 758$, it is easy to check that $b_{4,T}=2$ is the unique case when $b_{n,T} \in \Integer$. However, the lower bound for spherical design of harmonic index $6$ is about $5.29 $, which is strictly larger than $b_{4,T}$.

\subsection{Non-existence of tight spherical harmonic designs of index $\{8,2\}$}
We want to determine $f_2$ such that
\[Q_{n,8}(x)+f_2Q_{n,2}(x)=a(x^2-\alpha^2)^2(x^2-\beta^2)^2-c_{n,T}.\]
From calculation we have the following results:
\[f_2=\frac{(n-2)(n+4)(n+5)(n+6)(n+14)}{90(n+12)^2}.\]
\[\alpha^2,\beta^2=\frac{7}{n+12} \pm \frac{\sqrt{42(n+5)(n+10)}}{(n+10)(n+12)}.\]
\[c_{n,T}=\frac{(n+2)(n+4)(n+5)(n+8)(n+14)(n^3+27n^2+356n-240)}{240(n+10)(n+12)^3}.\]
Then we have
\[ \begin{array}{rcl}
b_{n,T} & = & {\displaystyle \frac{n(n+6)(n+5)(n^2+15n+8)^2}{168(n^3+27n^2+356n-240)}} \\
& & \\
&=& {\displaystyle \frac{1}{168}(n^4+14n^3-133n^2+2638n-10584)} \\
& & \\
& & \qquad\qquad\qquad {\displaystyle +\frac{-4032n^2+26208n-15120}{n^3+27n^2+356n-240}}.
\end{array} \]
Let $p(n)=-4032n^2+26208n-15120$ and $q(n)=n^3+27n^2+356n-240$.
Then $|\frac{p(n)}{q(n)}| <\frac{1}{168}$ gives a condition for $b_{n,T}$ not being integer.
%$\frac{p(n)}{q(n)} <\frac{1}{168}$ if $n \leq 6$ and $\frac{p(n)}{q(n)} > -\frac{1}{168}$ if $1 \leq n \leq 5$ or $n \geq 677343$.
This implies that $b_{n,T}$ can not be integer if $n \geq 677343$.
We can check the remaining cases where $b_{n,T}$ is integer, and obtain
$b_{2,T}=2$, $b_{4,T}=9$, $b_{9,T}=96$. (The case $b_{2,T}=2$ is eliminated by Remark~\ref{remark:4.2}.)\\
If $n=4$, the SDP upper bound for $4$-distance set is $8.9981$.
If $n=9$, then $b_{n,T} \geq  2\binom{n+1}{2}$.  But neither $\frac{1-\alpha^2}{\beta^2-\alpha^2}$ nor $\frac{1-\beta^2}{\alpha^2-\beta^2}$ is an integer. By Lemma~\ref{lem:rat}, the case when $n=9$ is also impossible.

\subsection{Non-existence of tight spherical designs of harmonic index $\{10,6,2\}$.} \label{sec:10-6-2}

Consider the case when $Q_{n,10}(x)+f_6Q_{n,6}(x)+f_2Q_{n,2}(x)$ takes minimum value $-c_{n,T}$ at three non-negative points $\{0,\alpha,\beta\}$, i.e.,
\[ L(x) = Q_{n,10}(x)+f_6Q_{n,6}(x)+f_2Q_{n,2}(x)  =  ax^2(x^2-\alpha^2)^2(x^2-\beta^2)^2-c_{n,T}.\]
We can solve $f_6$ and $f_2$ as follows.
\[ \begin{array}{l}
{\displaystyle f_6=\frac{(n-2)(n+8)(n+18)(13n+28)}{1344(n-8)(n+16)}}, \\
\\
{\textstyle f_2=\frac{(n-2)(n+4)(n+8)(n+14)(n+18)(37n^3-742n^2+1792n+20256)}{129024(n-8)^2(n+12)(n+16)}}.
\end{array} \]
Then we can get $\alpha^2,\beta^2$ and the lower bound $b_{n,T}$.
\[ \begin{array}{l}
{\displaystyle \alpha^2,\beta^2=\frac{45(n-8)(n+12) \pm \sqrt{15(n-8)(n+12)(43n^2-244n-1952)}}{4(n-8)(n+12)(n+16)}},\\
\\
{\displaystyle b_{n,T}=\frac{ n (n+4)(n+8) (4 n^3-10 n^2-143 n-84)^2}{45(781n^4-9548n^3+10128n^2+160960n-108032)}}.
\end{array} \]
The following theorem is very useful to consider the existence of spherical design of harmonic index $\{10,6,2\}$.
\begin{theorem}[{\cite[Theorem 5.3]{N2011}}]\label{theo:7.3}
Let $X$ be an antipodal $s$-distance set on $\Sphere^{n-1}$ where $s$ is an even integer at least $4$. Let $I(X)=\{-1, \beta_1=0, \pm\beta_2, \ldots, \pm\beta_{\frac{s}{2}} \}$.
\begin{enumerate}[(1)]
\item
 Let $N=\binom{n+s-3}{s-2}$. If $\left| X \right| \geq 4N$, then for each $i = 1,\dots,s/2$,
\[ k_i := \prod_{j=1,\dots,\frac{s}{2}, j  \neq i} \frac{1 - \beta_j^2}{\beta_i^2 - \beta_j^2} \]
must be an integer with $\left| k_i \right| \leq U(N)$.\\
\item
 Let $N=\binom{n+s-4}{s-3}$. If $\left| X \right|\geq 4N+2$, then for each $i = 1,\dots,s/2$,
\[ k_i := \frac{1}{\beta_i}\prod_{j=2,\dots,\frac{s}{2}, j  \neq i} \frac{1 - \beta_j^2}{\beta_i^2 - \beta_j^2} \]
must be an integer with $\left| k_i \right| \leq \lfloor \sqrt{2N^2/(N+1)} \rfloor $.
\end{enumerate}
\end{theorem}

Similarly, tight spherical design of harmonic index $\{10,6,2\}$ is regarded as a spherical $5$-distance set $Y$ with $I(Y)=\{0, \pm\alpha, \pm\beta\}$. We construct an antipodal $6$-distance set $X^{\prime}=Y \cup (-Y)$ with $I(X^{\prime})=\{-1,0, \pm\alpha, \pm\beta\}$.\\

\noindent \textbf{Case (1)}:
 If $n \geq 170$, then $\left| Y \right| \geq 2\binom{n+6-3}{6-2}+1$.
Theorem~\ref{theo:6.3} implies that $\alpha,\beta \in \Rational$. Namely, $\alpha\beta=\sqrt{\frac{(n-8)(n+12)(n+16)}{15(23n-172)}} \in \Rational$. There exists an integer $u$ such that
$$15(n-8)(n+12)(n+16)(23n-172)=u^2.$$
Assume that $n-8 = Ay_1^2$, $n+12 = By_2^2$, $n+16 = Cy_3^2$, $23n-172 = Dy_4^2$. Let $p(n_1,n_2)$ be the prime divisor of $\gcd(n_1,n_2)$. Since $\gcd(n-8,n+12)$ can be one of $1,2,4,5,10,20$, then $p(n-8,n+12)=2$ or $5$. And with similar argument, we obtain the following result.
\[ \begin{array}{ll}
p(n-8,n+12)=2\ \text{or} \ 5, & p(n-8,n+16)=2 \ \text{or} \ 3, \\
\\
  p(n-8,23n-172)=2 \ \text{or} \ 3, & p(n+12,n+16)=2,  \\
\\
   p(n+12,23n-172)=2 \ \text{or} \ 7, & p(n+16,23n-172)=2, \ 3, \text{or} \ 5.
\end{array} \]
Then we get $16$ elliptic equations
\[ (n-8)(n+12)(n+16)=E y^2 \]
with $E \mid 2 \times 3 \times 5 \times 7$. We can obtain the integral solutions of these equations. \\
Multiply both sides of the equation $(n-8)(n+12)(n+16)=E y^2$ by $E^3$ and make linear transformation $M=E^2y$ and $N=E(n+12)$. Then $(n-8)(n+12)(n+16)=E y^2$ becomes $N(N-20E)(N+4E)=M^2$ with coefficients of $M^2$ and $N^3$ being $1$.\\
If $E=2$, with the linear transformation $N=2n+24$ and $M=4y$, the equation becomes $N(N-40)(N+8)=M^2$. It has only three integral solutions $(N,M)=(0,0)$, $(40,0)$ and $(-8,0)$.  \\
In the table below, we list all the cases when the equation $(n-8)(n+12)(n+16)=E y^2$ has non-trivial integral solutions.

\[ \begin{array}{c|l}
E & \text{solution}(n,y) \\ \hline%&
5 &  (33,105)  \\  \hline
7  & (16,32),(68,240) \\ \hline
6  & (20,48), (48,160)  \\ \hline
10  &  (24,48),(38,90)  \\ \hline
15  &  (308,1440)  \\ \hline
21   & (9,5),(488,2400)  \\ \hline
42   &  (12,8),(128,240)  \\ \hline
 210   &  (44,24) \\ \hline
\end{array} \]
There are two values for $n \geq 170$, i.e., $n=308,488$. But in these two cases, $\alpha \notin \Rational$.

\noindent \textbf{Case (2)}:
If $19 \leq n \leq 169$, then the second statement in Theorem \ref{theo:7.3} implies that $\frac{1-\beta^2}{\alpha(\alpha^2-\beta^2)}$ must be integer. And it is easy to check this can not be satisfied.

\noindent \textbf{Case (3)}:
 If $2 \leq n \leq 18$, then $b_{8,T}=8$ is the only integral case. Let
\[ m_1=\alpha^2+\beta^2, ~ m_2=(\alpha^2+\beta^2)^2+2\alpha^2\beta^2, ~ m_3=\alpha^2\beta^2(\alpha^2+\beta^2). \]
Then $m_1(m_2-m_1^2)-2m_3=0$. However, comparing the coefficients in $L(x)$ and $G(x)$, we obtain
\[
{\textstyle
\frac{1}{2}m_1(m_2-m_1^2)-m_3=\frac{1}{(n+12)(n+14)(n+16)^2} \left( \frac{18900(n-8)}{(n+8)(n+18)}f_6-\frac{225(13n+28)(n-2)}{16(n+16)} \right).}
\]
If $n=8$, then the RHS of the above equation is $-\frac{15}{8192} \neq 0$.
This is a contradiction. And it implies that, when $n=8$, we can not find a function $L(x)$ satisfying our assumption.

\subsection{Tight spherical designs of harmonic index $\{12,8,4\}$.}
Consider the case when $Q_{n,12}(x)+f_8Q_{n,8}(x)+f_4Q_{n,4}(x)$ takes minimum value $-c_{n,T}$ at three positive points $\{\alpha,\beta,\gamma\}$, i.e.,
\[
\begin{array}{l}
L(x)=Q_{n,12}(x)+f_8Q_{n,8}(x)+f_4Q_{n,4}(x)\\
\\
\qquad =a(x^2-\alpha^2)^2(x^2-\beta^2)^2(x^2-\gamma^2)^2-c_{n,T}=G(x).
\end{array}
\]
Let $Z_1=\alpha^2+\beta^2+\gamma^2$, $Z_2=\alpha^2\beta^2+\alpha^2\gamma^2+\beta^2\gamma^2$ and $Z_3=\alpha^2\beta^2\gamma^2$. Comparing the coefficients in $L(x)$ and $G(x)$, we can obtain the following results.
$$m_1=Z_1, \ m_2=Z_1^2+2Z_2,\ m_3=Z_1Z_2+Z_3,\ m_4=Z_2^2+2Z_1Z_3,\ m_5=Z_2Z_3.$$
where $-2am_1$, $am_2$, $-2am_3$, $am_4$ and $-2am_5$ are the coefficients of $x^{10}$, $x^8$, $x^6$, $x^4$, $x^2$ in polynomial $G(x)$, respectively.
 Note that $m_i$ is linear combination of $f_8$ and $f_4$. Substituting $Z_1=\frac{m_2-m_1^2}{2}$ and $Z_3=m_3-Z_1Z_2=m_3-m_1(\frac{m_2-m_1^2}{2})$ into $m_4$ and $m_5$, we can solve $f_8$ and $f_4$ as follows.
\[ \begin{array}{l}
{\displaystyle f_8=\frac{(n+8)(n+9)(n+22)}{180(n+20)}},\\
\\
{\displaystyle f_4=\frac{(n+4)(n+5)(n+8)(n+9)(n+18)(n+22)}{7200(n+16)(n+20)}}.
\end{array} \]
Immediately we have
\[ \begin{array}{l}
{\displaystyle Z_1=\alpha^2+\beta^2+\gamma^2=\frac{33}{n+20}},\\
\\
{\displaystyle Z_2=\alpha^2\beta^2+\beta^2\gamma^2+\alpha^2\gamma^2=\frac{231}{(n+16)(n+20)}},\\
\\
{\displaystyle Z_3=\alpha^2\beta^2\gamma^2=\frac{231}{(n+12)(n+16)(n+20)}},\\
\\
{\displaystyle b_{n,T}=\frac{1}{27720}(n+1)(n+2)(n+5)(n+6)(n+9)(n+10)}.
\end{array} \]

\noindent \textbf{Case (1)}:
If $n \geq 439$, then $b_{n,T} \geq 2\binom{n+7-3}{7-2}+1$. Theorem \ref{theo:6.3} implies that
\[ \alpha \beta \gamma=\sqrt{\frac{231}{(n+12)(n+16)(n+20)}} \in \Rational .\]
Equivalently, we have $(n+12)(n+16)(n+20)=231y^2$ for some integer $y$.
With linear transformation $M=231^2y$ and $N=231(n+16)$, we have $N(N^2-462^2)=M^2$.  It has integral solutions
$(N,M)=(-528,17424)$, $(-252,14112)$, $(1617,53361)$, $(3388,189728)$. However, this gives no positive integral solution for $n=N/231-16$.

 \noindent \textbf{Case (2)}:
 If $34 \leq n \leq 438$, then $b_{n,T} \geq 2\binom{n+7-4}{7-3}+1$. By Theorem~\ref{theo:7.3}, $\frac{(1-\beta^2)(1-\gamma^2)}{\alpha(\alpha^2-\beta^2)(\alpha^2-\gamma^2)}$ must be integer. And it is easy to check this can not be satisfied.

 \noindent \textbf{Case (3)}:
 If $n \leq 33$, then $b_{5,T}=35$, $b_{6,T}=64$, $b_{9,T}=285$, $b_{13,T}=1311$, $b_{16,T}=3315$, $b_{20,T}=9425$, $b_{23,T}=18560$ are the integral cases.
  But LP upper bounds for spherical $6$-distance set corresponding to $n=5,6,20,23$ are $30.2656$, $59.8173$, $9405.11$, $17926.1$, respectively.
  When $n=9,13,16$, we set up the SDP method on the upper bounds for spherical $6$-distance sets.
  However, the upper bounds coincide with our lower bounds $b_{n,T}$.
  We conclude that there exists no tight spherical design of harmonic index $\{12,8,4\}$ if $n \neq 9,13,16$.

\section{Concluding remarks}
\label{sec:remark}
\noindent
In this paper we considered mainly spherical designs of harmonic index $T=\{t\},$ or $T=\{t_1,t_2\}$.
For some $T=\{t_1,t_2,\ldots ,t_{\ell}\}$ (with $t_1=2e>t_2> \cdots > t_{\ell}$ and all $t_i$ are even), it seems that the general interesting case is where $L(x)=Q_{n,2e}(x)+f_{t_2}Q_{n,t_2}(x)+f_{t_3}Q_{n,t_3}(x)+\cdots +f_{t_{\ell}}Q_{n,t_{\ell}}(x)$ and the minimum value of $L(x)$ is at $\ell$ non-negative points $\alpha_1, \alpha_2,  \ldots , \alpha_{\ell}$.
Thus, further studies along this line would be interesting.

As we have shown in Section~\ref{sec:asymptotic} (as well as in previous sections), it seems remarkable that a spherical design of harmonic index $t=2e$ has a Fisher type lower bound $\left| Y \right|\geq$ (constant)$\cdot n^e$, which is the same order as for spherical $2e$-design.
So, all harmonic index $T$-designs are between harmonic index $2e$-designs and spherical $2e$-designs.
It seems that considering tight $T$-designs have some meaning, although it seems tight harmonic index $T$-designs rarely exist.

As it is discussed in Bannai--Okuda--Tagami \cite[Proposition 4.1]{BOT2015} that to some extent,
\[b_n := \lim_{e\rightarrow \infty} b_{n,2e}\]
is also studied.
(The result is explained in terms of Bessel functions.)
As some special cases are mentioned in \cite[p. 10]{BOT2015}, $b_n$ becomes greater that $n(n+1)/2$ if $n\geq 7$.
This implies that tight spherical designs of harmonic index $2e$ do not exist in general, if $t=2e$ become large, say for $n\geq 7$.
(This is proved rigorously for all $n\geq 7,$ although it is not proved in \cite{BOT2015}.)
On the other hand if $n\leq 6$, it seems possible to determine (i.e., to show the non-existence of tight designs in these cases, but it is not clear how we can show the non-existence of such harmonic index $2e$-designs whose size are close to the Fisher type lower bound.
It seems that this remains as an interesting open problem.

In concluding this paper, we remark that the theory (as well as the concept) of harmonic index $T$-designs in Q-polynomial association schemes exactly go parallel with the spherical case.
The concept of $T$-design for an arbitrary subset $T$ of the index set of nontrivial relations $\{1,2, \ldots ,d\} $ is already defined in Delsarte \cite[Section 3.4]{D1973} (1973).
On the other hand, it seems that any systematic study on some specific choices of $T$, beyond the case $T=\{1,2, \ldots t\}$ has not been studied, even for the case $T=\{t\}$.
We hope to discuss more on this topic in a separate paper.

\section*{Acknowledgments}
We thank Takayuki Okuda and Makoto Tagami for many helpful discussions on the earlier stage of the study of spherical designs of harmonic index $T$, in particular when $T=\{8,4\}.$ We thank Ziqing Xiang, Shunichi Yokoyama, Masanobu Kaneko, William Stein and Maoshen Xiong for the discussions of the elliptic diophantine equations, and the help how to use the database. Eiichi Bannai is supported in part by NSFC grant No. 11271257.


\begin{thebibliography}{99}

\bibitem{BV2008}
{\sc C. Bachoc, F. Vallentin},
New upper bounds for kissing numbers from semidefinite programming,
J. Amer. Math. Soc. 21 (2008), no. 3, 909--924.

\bibitem{BOT2015}
{\sc E. Bannai, T. Okuda, M. Tagami},
Spherical designs of harmonic index $t$,
J. Approx. Theory 195 (2015), 1--18.

\bibitem{BY2013}
{\sc A. Barg, W.-H. Yu},
New bounds for spherical two-distance sets,
Exp. Math., 22 (2013), no. 2, 187--194.

\bibitem{BDHSS2010}
{\sc P. G. Boyvalenkov, P. D. Dragnev, D. P. Hardin, E. B. Saff, M. M. Stoyanova},
  Universal lower bounds for potential energy of spherical codes,
 arXiv: 1503.07228v1.

\bibitem{D1973}
{\sc P. Delsarte},
{\it An algebraic approach to the association schemes of the coding theory},
Thesis, Universite Catholique de Louvain (1973)
Philips Res. Repts Suppl. 10 (1973).

\bibitem{DGS1977}
{\sc P. Delsarte, J. M. Goethals, J. J. Seidel},
Spherical codes and designs,
Geom. Dediacata 6 (1977), 363--388.

\bibitem{DS1989}
{\sc P. Delsarte, J. J. Seidel},
Fisher type inequalities for Euclidean $t$-designs,
Linear Algebra Appl., 114/115 (1989), 213--230.


\bibitem{DN2010}
{\sc Dimitar K. Dimitrov, Geno P. Nikolov},
 {\em Sharp bounds for the extreme zeros of classical orthogonal polynomials},
 J. Approx. Theory 162 (2010), no. 10, 1793--1804.

\bibitem{EL1992}
{\sc A. Elbert, A. Laforgia},
 {\em Asymptotic formulas for ultraspherical polynomials $P^{\lambda}_n(x)$ and their zeros for large values of $\lambda$},
 Proc. Amer. Math. Soc.  114 (1992), no. 2, 371--377.

\bibitem{GL2003}
{\sc C. Giordano, A. Laforgia},
 {\em On the Bernstein-type inequalities for ultraspherical polynomials}, Proceedings of the 6th International Symposium on Orthogonal Polynomials, Special Functions and their Applications (Rome, 2001). J. Comput. Appl. Math.  153 (2003), no. 1--2, 243--248.

\bibitem{LRS1977}
{\sc D. G. Larman, C. A. Rogers, J. J. Seidel},
On $2$-distance sets in Euclidean space,
Bull. London Math. Soc., 9 (1977), no. 3,  261--267.

\bibitem{LS1973}
{\sc P. W. H. Lemmens, J. J. Seidel},
Equiangular lines,
J. Algebra, 24 (1973), 494--512.

\bibitem{LV1993}
{\sc Y. I. Lyubich, L. N. Vaserstein},
Isometric embeddings between classical Banach spaces, cubature formulas, and spherical designs,
Geom. Dedicata 47 (1993), no. 3, 327--362.

\bibitem{M2009}
{\sc O. R. Musin},
Spherical two-distance sets,
J. Combin. Theory Ser. A., 116 (2009), no. 4, 988--995.

\bibitem{N2011}
{\sc H. Nozaki},
A generalization of Larman-Rogers-Seidel's theorem,
Discrete Math., 311 (2011), 792--799.

\bibitem{OY2015}
{\sc T. Okuda, W.-H. Yu},
 A new relative bound for equiangular lines and nonexistence of tight spherical designs of harmonic index 4. (submitted)
( A different title on arXiv: 1409.6995v1.)

\bibitem{SAGE}
SageMath, \verb@http://www.sagemath.org@.

\bibitem{S2003}
{\sc G. Szeg\"{o}},
Orthogonal Polynomials,
4th edition, Amer. Math. Soc., Colloquim Publications, volume 23 (2003).

\end{thebibliography}
\end{document}